\documentclass[preprint,onecolumn,12pt]{IEEEtran}
\usepackage[a4paper,bindingoffset=0in,left=0.82in,right=0.82in,top=0.9in,bottom=0.9in,footskip=.3in]{geometry}
\usepackage{blindtext}

\usepackage{hyperref}  

\usepackage{graphics}
\usepackage{graphicx}
\usepackage{epsfig}
\usepackage{amsmath} 
\usepackage{amssymb}
\usepackage{color}
\usepackage{amsfonts}
\usepackage{lineno} 
\usepackage{enumerate}
\usepackage{multirow}
\usepackage{arydshln}
\usepackage{caption}
\usepackage{adjustbox}

\allowdisplaybreaks

\newtheorem{theorem}{Theorem}
\newtheorem{definition}[theorem]{Definition}

\newtheorem{lemma}[theorem]{Lemma}

\newtheorem{remark}[theorem]{Remark}

\newcommand{\abs}[1]{\left|#1\right|}

\def\field#1{\mathbb #1}%
\def\R{\field{R}}%
\def\Z{\mathbb{Z}}
\newcommand{\Rn}[1][n]{\R^{#1}}

\newcommand{\Rp}{\R_{\geq 0}}
\newcommand{\Rsp}{\R_{> 0}}
\newcommand{\Zp}{\Z_{\geq 0}}
\newcommand{\Zsp}{\Z_{> 0}}

\def\K{\mathcal{K}}%
\def\Kinf{\K_\infty}%
\def\KL{\mathcal{KL}}%

\def\PD{\mathcal{PD}}%
\let\ol=\overline%
\let\ul=\underline%

\ifCLASSINFOpdf
\else
\fi

\begin{document}

\title{Integral versions of input-to-state stability for dual-rate nonlinear sampled-data systems}


\author{Navid Noroozi$^\textrm{a}\qquad$\thanks{$^\textrm{a}$Faculty of Computer Science and Mathematics, University of Passau, Innstra$\ss$e 33, 94032 Passau, Germany, \texttt{navid.noroozi@uni-passau.de}} Seyed Hossein~Mousavi$^\textrm{b}\qquad$\thanks{$^\textrm{b}$Department of Aerospace Engineering, University of Ryerson, Toronto, Ontario, Canada, \texttt{seyedhossein.mousavi@ryerson.ca}} Horacio~J.~Marquez$^\textrm{c}\qquad$\thanks{$^\textrm{c}$Department of Electrical and Computer Engineering, University of Alberta, Edmonton, Alberta, Canada, \texttt{hmarquez@ualberta.ca}}
\thanks{The work of N.~Noroozi was partially supported by the Alexander von Humboldt Foundation.}
\thanks{The work was mostly done when S.~H.~Mousavi was at the University of Alberta.}}

\maketitle

\begin{abstract}
This paper presents versions of integral input-to-state stability and integral input-to-integral-state stability for nonlinear sampled-data systems, under the low measurement rate constraint. In particular, we compensate the lack of measurements using an estimator approximately reconstructing the current state. Interestingly, under certain checkable conditions, we establish that a controller that semiglobally practically integral input-to-(integral-) state stabilizes an approximate discrete-time model of a \emph{single-rate} nonlinear sampled-data system, also stabilizes the exact discrete-time model of the nonlinear sampled-data system in the same sense implemented in a \emph{dual-rate} setting. Numerical simulations are given to illustrate the effectiveness of our results.
\end{abstract}


\section{Introduction}

Sampled-data systems consist of a continuous-time plant/process controlled by a digital controller interfaced through analogue-to-discrete (A/D) and discrete-to-analogue (D/A) converters.
When the plant is nonlinear, analysis and design of sampled-data systems present significant challenges and has attracted significant research attention; e.g.,~\cite{Tanasa2016,nkk15,Ustunturk2012,Monaco2007,lna06}.

Sampled-data systems can be designed by \emph{emulation}~\cite{Owens.1990,Castillo.1997,Laila.2002} or in \emph{discrete-time}~\cite{Dochain.1984,Mareels.1992,Nesic.2006}.
In the emulation approach, a continuous-time controller is first designed for the continuous-time plant and then the controller is discretized using any available technique for numerical integration.
In the discrete-time approach, on the other hand, first the plant model is discretized and then a controller is designed for the resulting discrete-time plant.

In general, the emulation approach presents the disadvantage that requires high sampling rates in order to recover the properties of the analog design and guarantee stability and desired performance. 
The discrete-time approach can usually offer similar results at lower sampling rates compared to emulation-based controllers.
Unfortunately, however, finding the exact discrete-time model of a continuous-time plant requires solving an initial value problem whose solution does not exist in closed-form for any practical nonlinear plant model.
A remedy for this issue is to conclude stability properties and/or system performance of the exact discrete-time model from an approximate model of the system.
This treatment has received considerable attention over the last two decades; e.g.,~\cite{ntk99,nl02,nt04,vnh12}.
It should be noted that there have been other approaches such as~\cite{Tanasa2016,Monaco2007}, where smoothness conditions on the model describing the system is required.

In the discrete-time approach, it is usually assumed that all signals in the loop are sampled regularly at the same sampling rate.
Practical considerations, however, make the use of different sampling rates for inputs and outputs highly desirable or even mandatory.
In~\cite{pm04}, Polushin and Marquez consider the dual-rate sampled-data stabilization problem assuming low measurement rates and propose a multirate controller that approximates state trajectories between samples based on a fast numerical integration scheme.
In~\cite{lml08}, the work~\cite{pm04} is extended to include the effect of disturbances, using the concept of input-to-state stability (ISS)~\cite{Sontag.1989}.
Furthermore, extensions to multi-rate output feedback control problem are reported in~\cite{bmr15,Ustunturk2012}.

Major outcome of this body of work is the understanding of under what conditions stability properties of a digital design based on an approximate discrete-time model of a sampled-data system, are also true for the exact (not available to the designer) discrete-time model of the system.
Particularly relevant to this work is the preservation of the important notions of ISS, and integral input-to-state stability (iISS)~\cite{asw00}.
Both concepts are of fundamental importance in control.
Informally, ISS and iISS capture the notion that the system state remains small, regarding the initial conditions, provided that the input (ISS) or integral of the input (iISS) are small.
In~\cite{nl02}, the authors consider the single-rate sampled-data systems and provide sufficient conditions under which a discrete-time controller that input-to-state stabilizes an approximate discrete-time model of a nonlinear plant with disturbances also input-to-state stabilizes the exact discrete-time plant model.
In~\cite{lml08}, the authors consider low measurement frequency restriction and generalize the ISS idea to the dual-rate case.
Integral versions of ISS, called semiglobal practical integral input-to-state stability (SP-iISS) and semiglobal practical integral input-to-integral-state stability (SP-iIiSS), of single-rate sampled-data systems via approximate models are considered in~\cite{na02}.

In this paper, our main interest is the study of dual rate iISS nonlinear sampled-data systems.
As we will see later, this is a source of difficulties with which this work is concerned.
Specially, we use an approximate discrete-time model of the plant to estimate inter-sampled values used by the controller working at the higher rate.
Then we show SP-iI(i)SS for the exact discrete-time model of the dual-rate sampled-data based on checkable sufficient conditions obtained from a \emph{single-rate} approximate model of the sampled-data system.
As a matter of fact, analysis of a single-rate sampled-data system is, in general, much easier than that for its multi-rate counterpart.
Therefore, our results may be preferable to the designer as we conclude stability of a dual-rate system from a single-rate model of the system.
Eventually, we illustrate the effectiveness of our results via a numerical example.

The rest of this paper is organized as follows:
Section \ref{S:Notation} defines the notation used throughout the paper.
Problem statement is given in Section \ref{S:Problem-Statement} and the main results are provided in Section \ref{S:Main-Results}.
An illustrative example is given in Section \ref{S:Examples}.
Finally, concluding remarks are provided as the last section.
 
\section{Notation} \label{S:Notation}

Throughout the paper, $\Rp (\Rsp)$ and $\Zp(\Zsp)$ denote the nonnegative (positive) real and nonnegative (positive) integer numbers, respectively.
We denote the standard Euclidean norm by $\abs{\cdot}$.
A function $\rho \colon \Rp \to \Rp$ is positive definite ($\rho \in \PD$) if it is continuous, zero at zero and positive elsewhere.
A positive definite function $\alpha$ is of class $\K$ ($\alpha \in \K$) if it is  strictly increasing. 
It is of class $\Kinf$ ($\alpha \in \Kinf$) if $\alpha \in \K$ and also
$\alpha(s) \to +\infty$ if $s \to +\infty$.
A continuous function $\beta \colon \Rp \times \Rp \to \Rp$ is of class $\KL$ ($\beta \in \KL$), if for each  $s \geq 0$, $\beta(\cdot,s) \in \K$, and for each $r \geq 0$, $\beta (r,\cdot)$ is decreasing and $\lim_{s \to +\infty} \beta (r,s) \to 0$.
For a given function $w \colon \Rp \to \Rn[p]$, $w_{T}[k]$ represents the restriction of the function $w(\cdot)$ to the interval $[kT,(k+1)T]$, where $k \in \Zp$ and $T > 0$.
Given $\gamma \in \K$ and a measurable function $w \colon \Rp \to \R^{d}$, we define $\abs{w}_\gamma := \int_0^{+\infty} \gamma(\abs{w(s)}) \mathrm{d} s$.
We denote the set of all functions with $\abs{w}_\gamma < + \infty$ by $\mathcal{L}_\gamma$.
Also, if $\abs{w}_\gamma < r$ for some $r \in \Rsp$, we write $w \in \mathcal{L}_\gamma (r)$.
Similarly, let $w \colon \Rp \to \R^{d}$ be a measurable function such that $\abs{w}_{\infty} := \sup_{t \geq 0} \abs{w(t)} < + \infty$.
We denote the set of all such functions by $\mathcal{L}_\infty$. Also, if $\abs{w}_\infty < r$ for some $r \in \Rsp$, we write $w \in \mathcal{L}_\infty (r)$.

\section{Problem Statement} \label{S:Problem-Statement}
Consider the following plant model
\begin{align} \label{eq:e55}
\dot{x} (t) = f(x(t),u (t),w(t)) ,
\end{align}
where $x(t) \in \Rn, u(t) \in \R^{m}$ and $w(t) \in \R^{p}$ are the state, the control input and the disturbance input, respectively. The function $f : \Rn \times \R^{m} \times \R^{p} \to \Rn$ is locally Lipschitz and $f(0,0,0) = 0$. Moreover, we assume that $w(.)$ is measurable and locally essentially bounded. 

We consider the scenario where the plant (\ref{eq:e55}) is connected to a computer control via a zero-order hold (ZOH) and a sampler. Thus the control signal is held constant during the sampling intervals, that is, $u(t) = u(kT)$ for all $t \in [kT,(k+1)T), k \in \Zp$, where $T > 0$ is the sampling period. For simplicity, we abuse our notation and write $u(k) := u(kT)$ and $x(k) := x(kT)$. By $T$ and $T_s$, we, respectively, denote the sampling periods of the ZOH and the sampling devices. We assume that $T_s = \ell T$ for some integer $\ell \geq 1$.
Note that $\ell = 1$ implies the case of single-rate fashion which does not often hold in practice.
The exact discrete-time plant model can be obtained from (\ref{eq:e55}) as follows
\begin{align}
x(k+1) & = x(k) + \int_{kT}^{(k+1)T} { f (x(\tau),u(k), w(\tau)) } \mathrm{d} \tau =: F_T^e (x(k),u(k),w_T[k]). & \label{eq:exact-model}
\end{align}
Unfortunately, this requires solving the initial value problem \eqref{eq:exact-model} which does not have a closed-form solution, in most cases of interest.
Consistent with the literature on sampled-data systems, we use a family of approximate discrete-time models
\begin{align} \label{eq:app-model}
x(k+1) = F^a_{T,h} (x(k),u(k),w_T[k]) ,
\end{align}
where $h$ is the parameter of the numerical integration and used to enhance the accuracy between the approximate discrete-time model~\eqref{eq:app-model} and the exact discrete-time model~\eqref{eq:exact-model}.
Although one can simply take $h = T$ and use some classic approximation methods such as the Euler method, it is more appropriate to choose $h$ different from $T$, as shown in~\cite{nt04} for instance.

The mismatch between the exact discrete-time model~\eqref{eq:exact-model} and the approximate discrete-time model~\eqref{eq:app-model} is required to be small in the following sense.

\begin{definition}  \label{D:one-step-cons} (\cite{Arcak.2004,lml08})
The family $F_{T,h}^{a}$ is said to be one-step consistent with $F_{T}^{e}$ if for any real numbers $(\Delta_x,\Delta_u,\Delta_w)$ there exist a function $\rho \in \Kinf$ and $T^* > 0$ such that for each fixed $T \in (0,T^*)$ there exists $h^* \in (0,T]$ such that the following holds
\begin{align*}
& |F_T^e (x,u,w_T) - F_{T,h}^a (x,u,w_T)| \leq T \rho (h)
\end{align*}
for all $x \in \Rn$, $u \in \R^m$, $w \in \mathcal{L}_\infty$ with $\abs{x} \leq \Delta_x$, $\abs{u} \leq \Delta_u$, $\abs{w_T}_\infty \leq \Delta_w$ and all $h \in (0,h^*)$. $\hfill \Box$
\end{definition}
Sufficient checkable conditions under which the one-step consistency is guaranteed are given in~\cite{nl02,nt04}.

Assuming full access to the state $x$, we consider a family of controllers given by
\begin{align}
	& u (k) = u_{T,h} (x_c(k)), \;\;\;\;    u_{T,h} (0) =0 , & \label{eq:e04}
\end{align}
where 
\begin{equation} \label{eq:e02}
x_c (k) = \left\{ \begin{array}{l}
 x(k), \quad k = i \ell , i \in \Zp \\
  F^a_{T,h} (x_c(k-1),u(k-1),0) \textrm{  with } x_c(i\ell) = x(i\ell), \textrm{ otherwise}.
 \end{array} \right.
\end{equation}
In other words, the controller uses the sampled value of the state every $T_s$ seconds; however, to compensate for the lack of information in the inter-samples the controller uses estimated state values obtained from the disturbance-free plant model. It should be noted that the family of controllers~\eqref{eq:e04} and~\eqref{eq:e02} reduces to 
\begin{align} \label{cont-single-rate}
	& u (k) = u_{T,h} (x(k)), \;\;\;\;    u_{T,h} (0) =0 , & 
\end{align}
when $\ell =1$.

The following definition ensuring a so-called \emph{uniform local boundedness} of $u_{T,h}$~\cite{ntk99,lml08} is required for proofs of the main results (Theorems~\ref{thm:iISS} and~\ref{thm:iIiSS} below).

\begin{definition} \label{def:uni-loc-lip} (\cite{lml08})
The control law $u_{T,h}$ is said to be uniformly locally Lipschitz if for any $\Delta_x > 0$ there exist positive real numbers $T^*$ and $\tilde{L}$ such that for each fixed $T \in (0,T^*)$ there exists $h^* \in (0,T]$ such that 
\begin{align*}
& |u_{T,h} (x_2) - u_{T,h} (x_1)| \leq \tilde{L} |x_2 - x_1| 
\end{align*}
for all $h \in (0,h^*)$ and all $x_1 , x_2 \in \Rn$ with $\max\{|x_1|,|x_2|\} \leq \Delta_x$. $\hfill \Box$
\end{definition}

We conclude iISS of the dual-rate sampled-data system based on the following Lyapunov stability property of the approximate discrete-time model of the \emph{single-rate} sampled-data system \eqref{eq:app-model} and~\eqref{cont-single-rate} that is denoted by
\begin{align} \label{eq:family-app}
& x(k+1) = \mathcal{F}^a_{T,h} (x(k),w_T[k]) . &
\end{align}

\begin{definition} \label{def:lyapunov-sp-iiss}
The family of systems \eqref{eq:family-app} is Lyapunov semiglobally practically integral input-to-state stable (Lyapunov-SP-iISS) if there exist functions $\ul\alpha,\ol\alpha \in \Kinf$, $\alpha \in \mathcal{PD}$ and $\hat\gamma \in \K$ with the following properties:

For any positive real numbers $(\Delta_1,\Delta_2,\Delta_3,\delta_1)$ there exist $T^* > 0$ and $M > 0$ such that for any fixed $T \in (0,T^*)$ there exists $h^* \in (0,T]$ such that for all $h \in (0,h^*)$ there exists a continuous function $V_{T,h} \colon \R^n \to \Rp$ such that for all $x \in \R^{n}$ with $|x| \leq \Delta_1$ and all $w \in \mathcal{L}_\infty (\Delta_2) \cap \mathcal{L}_{\hat\gamma} (\Delta_3)$ the following hold
\begin{align}
& \underline{\alpha} (|x|) \leq V_{T,h}(x) \leq \overline{\alpha} (|x|) , & \label{eq:e08} \\
& V_{T,h}(\mathcal{F}^a_{T,h} (x,w_T)) - V_{T,h}(x) \leq T \left[ - \alpha(|x|) + \frac{1}{T} \int_{kT}^{(k+1)T} \hat\gamma(\abs{w(s)}) \mathrm{d}s + \delta_1 \right] ; & \label{eq:e09}
\end{align}
moreover, 
\begin{align} \label{eq:0909}
& |V_{T,h} (x_1) - V_{T,h}(x_2)| \leq M |x_1 - x_2| &
\end{align}
for all $x_1 , x_2 \in \R^{n}$ with $\max \{ |x_1| , |x_2| \} \leq \Delta_1$, all $T \in (0,T^*)$ and all $h \in (0,h^*)$.
The function $V_{T,h}$ is also called a SP-iISS-Lyapunov function for the family of systems~\eqref{eq:family-app}.\\
Specially, if the previous conditions hold with $\alpha \in \Kinf$, then the family of systems \eqref{eq:family-app} is Lyapunov semiglobally practically integral input-to-integral-state stable (Lyapunov-SP-iIiSS) and the corresponding function $V_{T,h}$ is called a SP-iIiSS-Lyapunov function for the family of systems \eqref{eq:family-app}.$\hfill \Box$
\end{definition}
We note that special case of Definition~\ref{def:lyapunov-sp-iiss}, where $h = T$ has been given in~\cite{na02}.

Denote the exact discrete-time model of the dual-rate sampled-data system \eqref{eq:exact-model},~\eqref{eq:e04} and \eqref{eq:e02} by
\begin{align} \label{eq:family-exact}
& x(k+1) = \mathcal{F}^e_{T,h} (x(k),w_T[k]) . &
\end{align}
We are interested in stability properties of the family of systems \eqref{eq:family-exact}, that are defined below.
\begin{definition} \label{def:sp-iiss}
The family of systems \eqref{eq:family-exact} is said to be semiglobally practically integral-input-to-state stable (SP-iISS) if there exist $\alpha \in \Kinf, \beta \in \mathcal{KL}$ and $\gamma \in \K$ such that  for any positive real numbers $(\Delta_{x},\Delta_{w_1},\Delta_{w_2},\delta)$, there exists $T^* > 0$ such that for each $T \in (0,T^*)$ there exists $h^* \in (0,T]$ such that for all $w \in \mathcal{L}_\infty (\Delta_{w_1}) \cap \mathcal{L}_\gamma (\Delta_{w_2})$,  all $x(0) \in \Rn$ with $|x(0)| \leq \Delta_{x}$ and all $h \in (0,h^*)$, each solution to \eqref{eq:family-exact} exists and satisfies
\begin{align*}
& \alpha(\abs{x(k)}) \leq \beta (|x(0)|,k) + \int_{0}^{kT} \gamma(\abs{w(s)}) \mathrm{d}s + \delta  &
\end{align*}
for all $k \in \Zp$. $\hfill \Box$
\end{definition}

\begin{definition}
The family of systems \eqref{eq:family-exact} is said to be semiglobally practically integral-input-to-integral-state stable (SP-iIiSS) if there exist $\alpha , \chi \in \Kinf$ and $\gamma \in \K$ such that for any positive real numbers $(\Delta_x,\Delta_w,\delta)$, there exists $T^* > 0$ such that for each $T \in (0,T^*)$ there exists $h^* \in (0,T]$ such that for all $w \in \mathcal{L}_\infty (\Delta_w)$, all $x(0) \in \Rn$ with $|x(0)| \leq \Delta_{x}$ and all $h \in (0,h^*)$, each solution to \eqref{eq:family-exact} exists and satisfies
\begin{align*}
& T \sum_{i = 0}^{k-1} \alpha(\abs{x(i)}) \leq \chi (|x(0)|) + \int_{0}^{kT} \gamma(\abs{w(s)}) \mathrm{d}s + T k \delta  &
\end{align*}
for all $k \in \Zp$. $\hfill \Box$
\end{definition}

\begin{remark}
Regarding Definition~\ref{def:sp-iiss}, different versions of semiglobal iISS have been also reported in~\cite{Angeli.2000}, where the gain functions $\beta$ and/or $\gamma$ \emph{depend} on the positive real numbers $(\Delta_x,\Delta_w)$. $\hfill \Box$
\end{remark}

\section{Main Results} \label{S:Main-Results}

This section provides sufficient conditions for SP-iI(i)SS of the family of systems \eqref{eq:family-exact}.
In particular, our first contribution establishes SP-iISS of the exact discrete-time models \eqref{eq:family-exact} of the dual-rate sample-data under checkable conditions obtained from the single-rate approximate discrete-time model of the system.
\begin{theorem} \label{thm:iISS}
Assume that the following hold
\begin{enumerate}[i)]
    \item\label{item:iISS-app} The family of the approximate discrete-time models \eqref{eq:family-app} is Lyapunov-SP-iISS;
	\item\label{item:iISS-one-step} $F_{T,h}^a$ is one-step consistent with $F_T^e$;
	\item\label{item:iISS-lips} $u_{T,h}$ is uniformly locally Lipschitz.
\end{enumerate}
Then the family of exact discrete-time models \eqref{eq:family-exact} is SP-iISS.
\end{theorem}
\begin{proof}
Let positive real numbers $(\Delta_{x},\Delta_{w_1},\Delta_{w_2},\delta)$ be given.
Also, let $\alpha \in \mathcal{PD}$ and $\hat\gamma \in \mathcal{K}$ come from item~\ref{item:iISS-app}) of the theorem.
According to~\cite[Lemma IV.1]{asw00}, there exist $\tilde\rho_1 \in \Kinf$ and $\tilde\rho_2 \in \mathcal{L}$ such that $\alpha(\cdot) \geq \tilde\rho_1 (\cdot) \tilde\rho_2 (\cdot)$.
Denote $\rho_1 (\cdot) := \tilde\rho_1 \circ \overline\alpha^{-1} (\cdot)$ and $\rho_2 (\cdot) := \tilde\rho_2 \circ \underline\alpha^{-1} (\cdot)$.
Take $\Delta_1 := \underline\alpha^{-1} (\overline\alpha (\Delta_{x}) + \Delta_{w_2} + \delta) + 1$, $\Delta_2 := \Delta_{w_1}$, $\Delta_3 := \Delta_{w_2}$ and $\delta_1 := \frac{\rho_1(\delta/2)\rho_2(\Delta_1)}{4}$.
Let $\Delta_1$ generate $T^*_1$, $h^*_1$ and $\tilde L$ from item~\ref{item:iISS-lips}).
Let the data $(\Delta_1,\Delta_2,\Delta_3,\delta_1)$ generate $T^*_2$, $ h^*_2$, $M$ and a SP-iISS Lyapunov for the family $\mathcal{F}_{T,h}^a$.
By slight abuse of notation, denote $V_k := V_{T,h} (x (k))$, $V^e_{k+1} := V_{T,h} (\mathcal{F}^e_{T,h} (x (k) , w_T[k]))$ and $V^a_{k+1} := V_{T,h} (\mathcal{F}^a_{T,h} (x(k), w_T[k]))$ for all $k\in\Zp$.
Take any $\varepsilon_1 > 0$.
Let $(\Delta_1,\Delta_2,\varepsilon_1)$ generate $T^*_3$, $h^*_3$ and $L$ from Lemma~\ref{L:02} (see~\ref{app:technical-lemmas}).
By item~\ref{item:iISS-lips}), there exist $T^*_4>0$, $h^*_4>0$ and $\Delta_u>0$ such that $\abs{u_{T,h} (x)} \leq \Delta_u$ for all $x \in \Rn$ with $\abs{x} \leq \Delta_1$, all $T \in (0,T^*_4)$ and all $h \in (0,h^*_4)$.
Also, let $(\Delta_1,\Delta_u,\Delta_2)$ generate $\rho$, $T^*_5$ and $h^*_5$ from item~\ref{item:iISS-one-step}).
Take $T^*_6$ and $h^*_6$ such that
$$
M \left[ \rho(h^*_6) + \tilde{L} (e^{LT^*_6} - 1) \left(T^*_6 \varepsilon_1 + L e^{L(\ell-1)T^*_6} (\ell-1) \Delta_2 \right) \right] \leq \delta_1 .
$$
Pick $T^*_7$ and $h^*_7$ small enough such that
$$
T^*_7  \Big[\hat\gamma (\Delta_2) + \delta_1 + M \big[\rho(h^*_7) + \tilde{L} (e^{LT^*_7} - 1) \left(T^*_7 \varepsilon_1 + L e^{L(\ell-1)T^*_7} \Delta_3 \right) \big] \Big] \leq \frac{\delta}{2} .
$$
Let $\tilde\delta > 0$ be such that $\underline\alpha^{-1} (\overline\alpha (\Delta_{x}) + \Delta_3 + \delta + \tilde\delta) \leq \underline\alpha^{-1} (\overline\alpha (\Delta_{x}) + \Delta_3 + \delta ) + 1/2$. Let $T^*_8$ and $h_8^*$ be such that
$$
T^*_8  \Big[ \rho(h^*_8) + \tilde{L} (e^{LT^*_8} - 1) \left(T^*_8 \varepsilon_1 + (\ell-1) L e^{L(\ell-1)T^*_8} \Delta_2 \right) \Big] \leq \frac{1}{2} .
$$
Take $T^*_9$ such that $(\hat\gamma (\Delta_2) + \delta_1)T^*_9 \leq \tilde\delta$. Denote $T^* := \min \{ \min_{i\in\{1,\cdots,9\}} T^*_i , 1 \}$ and $h^*$ $:=$ $\min \{ \min_{i\in\{1,\cdots,8\}} h_i^*, T^* \}$.

Pick an arbitrary $x(k)$ such that $V_k \leq \ol\alpha (\Delta_x) + \Delta_3 + \delta$. It follows from the fact that the family of the approximate discrete-time model $\mathcal{F}_{T,h}^a$ is Lyapunov-SP-iISS that
\begin{align*}
V^e_{k+1} - V_k = & V^a_{k+1} - V_k + V^e_{k+1} - V^a_{k+1} & \\
\leq & T \left[ - \alpha(\abs{x(k)}) + \frac{1}{T} \int_{kT}^{(k+1)T} \hat\gamma(\abs{w(s)}) \mathrm{d}s + \delta_1 \right]  + \abs{V^e_{k+1} - V^a_{k+1}} . &
\end{align*}
From item~\ref{item:iISS-app}) of the theorem and our choice of $T^*_9$, we have
\begin{align*}
\underline{\alpha} (\abs{\mathcal{F}^a_{T,h}(x,w_T)}) & \leq V^a_{k+1} \leq V_k + T (\hat\gamma (\Delta_2) + \delta_1) < \bar{\alpha} (\Delta_x) + \Delta_3 + \delta + \tilde\delta . &
\end{align*}
So we get
\begin{align*}
\abs{\mathcal{F}^a_{T,h}(x,w_T)} & \leq \underline{\alpha}^{-1} (\bar{\alpha} (\Delta_x) + \Delta_3 + \delta + \tilde\delta ) .
\end{align*}
It follows from the choice of $\tilde\delta$ and the definition of $\Delta_1$ that
\begin{align} \label{eq:e20}
\abs{\mathcal{F}^a_{T,h}(x,w_T)} \leq \underline{\alpha}^{-1} (\bar{\alpha} (\Delta_x) + \Delta_3 + \delta) +\frac{1}{2} < \Delta_1 .
\end{align}
We note that
\begin{align*}
\abs{\mathcal{F}^e_{T,h}(x,w_T) - \mathcal{F}^a_{T,h}(x,w_T)} \leq &  \abs{\mathcal{F}^e_{T,h}(x,w_T) - \mathcal{\tilde{F}}^e_{T,h}(x,w_T)} & \\
& + \abs{\mathcal{\tilde{F}}^e_{T,h}(x,w_T) - \mathcal{F}^a_{T,h}(x,w_T)} , &
\end{align*}
where $\mathcal{\tilde{F}}^e_{T,h}(x,w_T)$ describes the exact discrete-time model of the single-rate sampled-data system~\eqref{eq:exact-model} and~\eqref{cont-single-rate}, that is, $x(k+1) = \mathcal{\tilde{F}}^e_{T,h} (x(k),w_T[k])$.
By application of the standard version of Gronwall inequality~\cite{Khalil.2002} and item~\ref{item:iISS-lips}) of the theorem, we have
\begin{align} \label{eq:fe-fes}
\abs{\mathcal{F}^e_{T,h}(x(k),w_T[k]) - \mathcal{\tilde{F}}^e_{T,h}(x(k),w_T[k])} \leq \tilde{L} (e^{LT} - 1) \abs{x(k) - x_c (k)} .
\end{align}
It follows from $T^*_4$, $h^*_4$, $T^*_5$ and $h^*_5$ that
\begin{align}\label{eq:fes-fas}
\abs{\mathcal{\tilde{F}}^e_{T,h}(x(k),w_T[k]) - \mathcal{F}^a_{T,h}(x(k),w_T[k])} \leq T \rho (h) .
\end{align}
Given~\eqref{eq:fe-fes} and~\eqref{eq:fes-fas}, we have
\begin{align} \label{eq:e22}
& \abs{\mathcal{F}^e_{T,h}(x(k),w_T[k]) - \mathcal{F}^a_{T,h}(x(k),w_T[k])} \leq T (\rho(h) + \tilde{L} (e^{LT} \! - \! 1) \abs{x(k)-x_c(k)} ) . &
\end{align}
Applying Lemma~\ref{L:02} to the right-hand side of (\ref{eq:e22}) yields
\begin{align*}
&\abs{\mathcal{F}^e_{T,h}(x,w_T) - \mathcal{F}^a_{T,h}(x,w_T)} & \\
& \leq T \left(\rho(h) + \tilde{L} (e^{LT} \! - \! 1) \left(T \varepsilon_1 + L \sum_{i=0}^{\ell-2} e^{L(i+1)T} \int_{(k-i - 1)T}^{(k-i)T}\abs{w(s)} \mathrm{d}s \right) \right)  & \nonumber \\
& \leq T \left(\rho(h) + \tilde{L} (e^{LT} \! - \! 1) \left(T \varepsilon_1 + L e^{L(\ell-1)T} \sum_{i=0}^{\ell-2} \int_{(k-i - 1)T}^{(k-i)T}\abs{w(s)} \mathrm{d}s \right) \right) . &
\end{align*}
It follows from the fact that $w \in \mathcal{L}_\infty(\Delta_2) \cap \mathcal{L}_{\hat\gamma}(\Delta_3)$ that
\begin{align*}
& \abs{\mathcal{F}^e_{T,h}(x,w_T) \! - \! \mathcal{F}^a_{T,h}(x,w_T)} \! \leq \! T \!\left(\!\rho(h) \! + \! \tilde{L} (e^{LT} \! - \! 1) \! \left(T \varepsilon_1 \!+\! (\ell-1) L e^{L(\ell-1)T} \Delta_2 \right) \! \right) \! . &
\end{align*}
From the choice of $T^*_8$ and $h_8^*$, we get
\begin{align} \label{eq:e21}
& \abs{\mathcal{F}^e_{T,h}(x,w_T) - \mathcal{F}^a_{T,h}(x,w_T)} \leq \frac{1}{2} . &
\end{align}
Exploiting the first inequality of (\ref{eq:e20}) and (\ref{eq:e21}) gives
\begin{align*}
\abs{\mathcal{F}^e_{T,h}(x,w_T)} & \leq \abs{\mathcal{F}^a_{T,h}(x,w_T)} + \abs{\mathcal{F}^e_{T,h}(x,w_T) - \mathcal{F}^a_{T,h}(x,w_T)} & \\
& \leq \underline{\alpha}^{-1} (\bar{\alpha} (\Delta_x) + \Delta_3 + \delta)+\frac{1}{2} + \frac{1}{2} . &
\end{align*}
From the definition of $\Delta_1$, we have
\begin{align} \label{eq:fe-delta}
& \abs{\mathcal{F}^e_{T,h}(x,w_T)} \leq \Delta_1 . &
\end{align}
By \eqref{eq:0909}, the second inequality of~\eqref{eq:e20} and~\eqref{eq:fe-delta}, we get
\begin{align*}
V^e_{k+1} - V_k \leq & T \left[ - \alpha(\abs{x(k)}) + \frac{1}{T} \int_{kT}^{(k+1)T} \hat\gamma(\abs{w(s)}) \mathrm{d}s + \delta_1 \right] \\
& + M \abs{\mathcal{F}^e_{T,h}(x,w_T) - \mathcal{F}^a_{T,h}(x,w_T)} .
\end{align*}
It follows from (\ref{eq:e22}) that
\begin{align*}
V^e_{k+1} - V_k \leq & T \left[ - \alpha(\abs{x(k)}) + \frac{1}{T} \int_{kT}^{(k+1)T} \hat\gamma(\abs{w(s)}) \mathrm{d}s + \delta_1 \right] & \nonumber\\
& + M \Big[ T\rho(h) + \tilde{L} (e^{LT} - 1) \abs{x(k)-x_c(k)} \Big].
\end{align*}
By Lemma~\ref{L:02}, we have
\begin{align*}
& V^e_{k+1} - V_k \leq T \left[ - \alpha(\abs{x(k)}) + \frac{1}{T} \int_{kT}^{(k+1)T} \hat\gamma(\abs{w(s)}) \mathrm{d}s + \delta_1 \right]&  \\
& \quad	+ M \Bigg[ T \rho(h) + \tilde{L} (e^{LT} - 1) \left(T \varepsilon_1 + L \sum_{i=0}^{\ell-2} e^{L(i+1)T} \int_{(k-i - 1)T}^{(k-i)T}\abs{w(s)} \mathrm{d}s \right) \Bigg] & \\
 &\leq T \left[ - \alpha(\abs{x(k)}) + \frac{1}{T} \int_{kT}^{(k+1)T} \hat\gamma(\abs{w(s)}) \mathrm{d}s + \delta_1 \right] \nonumber\\
& \quad+ M \Bigg[ T \rho(h) + \tilde{L} (e^{LT} - 1) \left(T \varepsilon_1 + L e^{L(\ell-1)T} \sum_{i=0}^{\ell-2} \int_{(k-i - 1)T}^{(k-i)T}\abs{w(s)} \mathrm{d}s \right) \Bigg] . &
\end{align*}
Given the fact that $w \in \mathcal{L}_\infty(\Delta_2) \cap \mathcal{L}_{\hat\gamma}(\Delta_3)$, we obtain
\begin{align*}
V^e_{k+1} - V_k \leq & T \left[ - \alpha(\abs{x(k)}) + \frac{1}{T} \int_{kT}^{(k+1)T} \hat\gamma(\abs{w(s)}) \mathrm{d}s + \delta_1 \right] \nonumber\\
& + T M \Bigg[ \rho(h) + \tilde{L} (e^{LT} - 1) \left( \varepsilon_1 + L e^{L(\ell-1)T} (\ell-1) \Delta_2 \right) \Bigg] .
\end{align*}
By the definitions of $T^*_6$ and $h^*_6$, we get
\begin{align*}
V^e_{k+1} - V_k \leq & T \left[ - \alpha(\abs{x(k)}) + \frac{1}{T} \int_{kT}^{(k+1)T} \hat\gamma(\abs{w(s)}) \mathrm{d}s + 2 \delta_1 \right] .
\end{align*}
It follows from the choice of $\tilde\rho_1(\cdot)$ and $\tilde\rho_2(\cdot)$, and the definitions of $\rho_1(\cdot)$ and $\rho_2(\cdot)$ that
\begin{align*}
V^e_{k+1} - V_k \leq & T \left[ - \rho_1(V_k) \rho_2 (V_k) + \frac{1}{T} \int_{kT}^{(k + 1)T} \hat\gamma(\abs{w(s)}) \mathrm{d}s + 2 \delta_1 \right] .
\end{align*}
Without loss of generality assume that $\ul\alpha$, coming from item~\ref{item:iISS-app}), satisfies $\ul\alpha (s) \leq s$ for all $s \geq 0$.
Then from the definition of $\Delta_1$ and the choice of $x(k)$, $\ul\alpha^{-1} (V_k) < \Delta_1$ implies $V_k < \ul\alpha(\Delta_1) \leq \Delta_1$.
It follows from the fact that $\rho_1(s) \rho_2 (s) \geq 4\delta_1$ for all $s \in [\delta/2 , \Delta_1]$ that
\begin{align}
V^e_{k+1} - V_k \leq & T \left[ - \frac{1}{2}\rho_1(V_k) \rho_2 (V_k) + \frac{1}{T} \int_{kT}^{(k + 1)T} \hat\gamma(\abs{w(s)}) \mathrm{d}s \right] & \label{eq:e10}
\end{align}
when $\Delta_1 \geq V_k \geq \frac{\delta}{2}$. Furthermore, with the choice of $T^*_7$ and $h^*_7$, we have
\begin{align}
V^e_{k+1} \leq V^a_{k+1} + \abs{V^e_{k+1} - V^a_{k+1}} \leq V_k + \frac{\delta}{2} . \label{eq:e11}
\end{align}
Define $w_k := \int_0^{kT} \hat\gamma (\abs{w(s)}) \mathrm{d}s$ and $y_k := V_k - w_k$. We emphasize that $w_0 = 0$ and $w_k$ is nondecreasing.
Denote $\tilde\rho(\cdot) := \frac{1}{2} \rho_1(\cdot)\rho_2(\cdot)$. 
From the definitions of $w_k$ and $y_k$,~\eqref{eq:e11} and the fact that $V^e_{k+1} - w_{k+1} \leq V^e_{k+1} - w_k$, we have
\begin{align*}
& y_{k+1} \leq y_k + \frac{\delta}{2} . &
\end{align*}
Moreover, $V_k \geq \frac{\delta}{2}$ implies that $y_k \geq \frac{\delta}{2} - w_k \geq \frac{\delta}{2}$. This together with (\ref{eq:e10}) gives
\begin{align*}
& y_{k+1} - y_k \leq - T \tilde\rho (\max \{ y_k + w_k , 0 \}) &
\end{align*}
whenever $y_k \leq \overline\alpha(\Delta_x) + \delta$.
It should be pointed out that $y_k \geq - \Delta_3$ for all $k \in \Zp$ because $V_k \geq 0$ and $w_k \leq \Delta_3$.
With the same arguments as those in \cite{na02}, for any $y_0 \in [ 0 , \overline\alpha(\Delta_x) + \delta]$ we have that $y_k \in [ - \Delta_3 , \overline\alpha(\Delta_x) + \delta ]$ for all $k \in \mathbb{Z}_{\geq 0}$.
This implies that $V_k \leq \ol\alpha (\Delta_x) + \Delta_3 + \delta$ for all $k \in \Zp$ and thus we can repeat all the above arguments for each $k$.
Moreover, all conditions of \cite[Lemma 3]{na02} are met with $\Delta_y = \overline\alpha(\Delta_x) + \delta$, $c_1 = c_2 = \frac{\delta}{2}$ and $k^* = +\infty$ therein. Therefore, there exists $\beta \in \mathcal{KL}$ such that the following holds
\begin{align*}
y_k \leq \beta (y_0 , kT) + w_k+ \frac{\delta}{2} + \frac{\delta}{2} \qquad \forall k \in  \Zp .
\end{align*}
It follows from the definitions of $y_k$ and $w_k$, the fact that $y_0 = V_0$, and the second inequality of \eqref{eq:e08} that
\begin{align*}
& \ul\alpha (\abs{x (k)}) \leq \beta (\ol\alpha(|x(0)|), kT) + 2 \int_{0}^{kT} \hat\gamma(\abs{w(s)}) \mathrm{d}s + \delta  & \end{align*}
for all $k \in \Zp$.
This completes the proof.
\end{proof}
Next, in the following theorem, we show SP-iIiSS of the exact discrete-time models \eqref{eq:family-exact}.
\begin{theorem} \label{thm:iIiSS}
Assume that the following hold
\begin{enumerate}[i)]
\item\label{item:iIiSS-app} The family of the approximate discrete-time models \eqref{eq:family-app} is Lyapunov-SP-iIiSS;
\item\label{item:iIiSS-one-step} $F_{T,h}^a$ is one-step consistent with $F_T^e$;
\item\label{item:iIiSS-lips} $u_{T,h}$ is uniformly locally Lipschitz.
\end{enumerate}
Then the family of exact discrete-time models \eqref{eq:family-exact} is SP-iIiSS.
\end{theorem}
\begin{proof}
Take any positive real numbers $(\tilde\Delta_x,\tilde\Delta_w,\tilde\delta)$.
Denote $\Delta_x := \ul\alpha^{-1} \big( \ol\alpha \circ \alpha^{-1} (\tilde\Delta_w + \tilde\delta) + 1\big)$, $\Delta_w := \max \{ \tilde\Delta_w , \hat\gamma^{-1} \circ \alpha (\tilde\Delta_x)\}$ and $\delta := \tilde\delta$.
Let the data $(\Delta_{x},\Delta_w,\delta)$ generate $T^*_1$, $h^*_1$ and the family of Lyapunov functions $V_{T,h}$ coming from Lemma~\ref{L:03} (see~\ref{app:technical-lemmas}).
Take $T^*_2$ such that $T^*_2 [\hat\gamma(\Delta_w) + \delta] < 1$.
Let $T^* := \min \{T^*_1 , T^*_2 \}$, $h^* := \min \{h^*_1,T^*\}$.
It follows from the second inequality of (\ref{eq:e08}) and the monotonicity of $\alpha$ that
\begin{align}
V_{T,h}(\mathcal{F}^e_{T,h} (x (k),w_T[k])) - V_{T,h}(x(k)) \leq T & \bigg[ - \alpha \circ \overline\alpha^{-1} (V_{T,h} (x(k))) & \nonumber\\
& \; + \frac{1}{T} \int_{kT}^{(k + 1)T} \hat\gamma(\abs{w(s)}) \mathrm{d}s + \delta \bigg] & \label{eq:e29}
\end{align}
for all $h \in (0,h^*)$ associated with each fixed $T \in (0,T^*)$, all $\abs{x(k)} \leq \Delta_{x}$ and all $w \in \mathcal{L}_\infty (\Delta_w)$.
Take an arbitrary $x(0)$ with $\abs{x(0)} \leq \tilde\Delta_x$.
It follows from the definition of $\Delta_w$ and the monotonicity of $\alpha$ that $V_{T,h} (x(0)) \leq \overline\alpha (\tilde\Delta_x) \leq \ol\alpha \circ \alpha^{-1} (\hat\gamma (\Delta_w) +\delta)$.
Now we show that for any $x(0)$ with $V_{T,h} (x(0)) \leq \ol\alpha \circ \alpha^{-1} (\hat\gamma (\Delta_w) +\delta)$, we have $\abs{x(k)} < \Delta_x$ for all $k \in \Zp$.
To see this, note that from~\eqref{eq:e29}, the fact that $w \in \mathcal{L}_\infty (\Delta_w)$ and the choice of $T^*_2$, we have
\begin{align*}
V_{T,h} (x(1)) & \leq V_{T,h} (x(0)) + T (\Delta_w + \delta)  < \ol\alpha \circ \alpha^{-1} (\hat\gamma (\Delta_w) +\delta) + 1 .
\end{align*}
It follows from the first inequality of (\ref{eq:e08}) and the definition of $\Delta_x$ that $\abs{x(1)} < \Delta_x$.
Then $x(1)$ is such that either $V_{T,h} (x(1)) \leq \ol\alpha \circ \alpha^{-1} (\hat\gamma (\Delta_w) +\delta)$ or $\ol\alpha \circ \alpha^{-1} (\hat\gamma (\Delta_w) +\delta) \leq V_{T,h} (x(1)) < \ol\alpha \circ \alpha^{-1} (\hat\gamma (\Delta_w) +\delta) + 1$.
In the former case, with the same arguments as above we see $V_{T,h} (x(2)) < \ol\alpha \circ \alpha^{-1} (\hat\gamma (\Delta_w) +\delta) + 1$.
In the latter case, it follows from~\eqref{eq:e29}, the fact that $w \in \mathcal{L}_\infty (\Delta_w)$ and the choice of $T^*_2$ that
\begin{align*}
V_{T,h} (x(2)) & \leq V_{T,h} (x(1)) - T \alpha \circ \overline\alpha^{-1} \circ \ol\alpha \circ \alpha^{-1} (\hat\gamma (\Delta_w) +\delta) + T (\hat\gamma(\Delta_w) + \delta)  & \\
& < \ol\alpha \circ \alpha^{-1} (\hat\gamma (\Delta_w) +\delta) + 1 .
\end{align*}
Repeating the same procedure yields $V_{T,h} (x(k))$ $<$ $\ol\alpha \circ \alpha^{-1} (\hat\gamma (\Delta_w) +\delta) + 1$ for all $k \in \Zp$.
This together with the first inequality of (\ref{eq:e08}) and the definition of $\Delta_x$ gives $\abs{x(k)} < \Delta_x$ for all $k \in \Zp$.
Therefore, (\ref{eq:e29}) holds for all $k \in \Zp$. 

Taking a sum of (\ref{eq:e29}) from $0$ to $k$ gives
\begin{align*}
 V_{T,h}(x (k)) - V_{T,h}(x(0)) & = \sum_{i=0}^{k-1} \Big[ V_{T,h} (x (i+1)) - V_{T,h} (x(i)) \Big] \nonumber\\
& \leq \! - T \sum_{i=0}^{k-1} \! \alpha \circ \ol\alpha^{-1} (V_{T,h}(x(i)))  \!+\! \int_{0}^{k T} \!\!\!\!\!\!\hat\gamma(\abs{w(s)}) \mathrm{d}s \!+ \!T k \delta   . &
\end{align*}
It follows from (\ref{eq:e08}) that
\begin{align*}
T \sum_{i=0}^{k-1} \alpha \circ \overline\alpha^{-1} \circ \ul\alpha(\abs{x(i)}) & \leq T \sum_{i=0}^{k-1} \alpha \circ \ol\alpha^{-1} (V_{T,h}(x(i))) & \nonumber\\
& \leq \ol\alpha(\abs{x(0)}) + \int_{0}^{k T} \hat\gamma(\abs{w(s)}) \mathrm{d}s + T k \delta . &
\end{align*}
This complete the proof.
\end{proof}

\begin{remark}
Note that all of our results are valid if we set $h = T$. This case not only because of its simplicity but also because several approximation methods (such as the classic Euler approximation) preserve the structure and types of nonlinearities of the continuous-time system and, hence, may be preferable to the designer. It should be noted that our results also hold for this special case. $\hfill \Box$
\end{remark}

\begin{remark}
Throughout the paper we the full information case in which all the state variables of the plant are measurable. This restriction can be relaxed using an output-based dynamic controller along the lines in~\cite{Ustunturk2012}. In this case, however, the results can be proved at the expense of possibly more restrictive assumptions including a uniform global Lipschitz condition on $F_T^e$.  $\hfill \Box$
\end{remark}

\begin{remark}
Here are two points which can be observed from the proofs of our results:
\begin{itemize}
\item Because of dealing with a multi-rate scheme, more complicated derivation and analysis are required in comparison with those in~\cite{na02}.
In particular, the mismatch between the plant state $x$ and the estimator state $x_c$ produces an additional term in equation~\eqref{eq:e22}, compared with its single-rate counterpart in proof of Theorem 1 in~\cite{na02}.
This term, in turn, leads to nontrivial difficulties in the following equations and analysis of the system.
\item Based on the definition of $T^*$ given in the proofs of our results, there is a trade-off between the values of $\ell$ and $T^*$.
In other words, the slower sampler (the larger $\ell$) is, the smaller maximum value should be chosen for $T$ to guarantee the properties of the system. $\hfill \Box$
\end{itemize}
\end{remark}

\section{Example} \label{S:Examples}
Consider the following continuous-time plant
\begin{equation}\label{ex:main}
	\begin{array}{l}
		\dot{x}_1 = -\frac{x_1}{1+x_1^2} + x_1 x_2, \\
		\dot{x}_2 = u + w ,
	\end{array}
\end{equation}
where $x=(x_1,x_2)$ is the state, $u$ is the control input and $w$ is the disturbance input that is given by 
	\begin{align}\label{ex:disturbance}
		w (t) = \left\{ \begin{array}{rcl}
			1 & & 2k \leq t < (2k+1) , \\
			-1 & & 2k+1 \leq t < 2(k+1) , \\
			0 & &  t \geq 6 ,
		\end{array}\right.
	\end{align}
for $k=0,1,2$.

Assume that the plant is located between a sampler and a ZOH device.
Given the case $h = T$, we use the following approximate model for controller design~\cite{nl02,Grune.2001}
\begin{align}\label{ex:plant-app-mod}
	\begin{array}{l}
		x_1 (k+1) = x_1(k) + T \left(-\frac{x_1(k)}{1+x_1(k)^2} + x_1(k) x_2(k) \right), \\
		x_2 (k+1)= x_2(k) + T u(k) + \int_{kT}^{(k+1)T} w(s) \mathrm{d}s . 
	\end{array}
\end{align}
The above model is one-step consistent with the corresponding exact model.
Following the Lyapunov-based controller redesign for sampled-data systems developed in \cite{Nesic.2005}, the control law is given by 
$u = u_{ct} + T u_1 , $
where $u_{ct}$ denotes a continuous-time controller designed based on the continuous-time model.
Moreover, the correction term $u_1$ is designed based on the approximate discrete-time model such that some better performance is achieved. In that way, we have
\begin{equation}\label{claww}
	u_{T,h}(x) = \frac{2x_1^2}{1+x_1^2} - c_1 x_2 - T c_2 x_2 x_1^2 ,
\end{equation}
where $c_1,c_2 >0$.
Note that the control law \eqref{claww} is uniformly locally Lipschitz. Let 
$$V (x_1,x_2) = 
\ln (1+x_1^2)+\frac{1}{2} x_2^2$$ 
be the SP-iISS Lyapunov candidate for the approximate model of the single-rate sampled-data system~\eqref{ex:plant-app-mod} and~\eqref{claww}. 
It is straightforward to see that for any positive real numbers $(\Delta_1,\Delta_2,\Delta_3,\delta)$, there exists $T^* > 0$ such that 
for all $T \in (0,T^*)$ and all $x \in \R^{n_x}$ with $|x| \leq \Delta_1$ the following holds
\begin{align*}
	V (x(k+1)) - V (x(k)) \leq T  \bigg[ & - \frac{2x_1^2}{1+x_1^2} - (c_1 - 1/2) x_2^2 - T c_2 x_2^2 x_1^2  +  \frac{1}{T} \int_{kT}^{(k+1)T} \!\!\!\!\!\!\!\!\abs{w(s)} \mathrm{d}s  + \delta_1 \bigg]  & 
\end{align*}
with $w(\cdot)$ defined by~\eqref{ex:disturbance}.
To evaluate the efficiency, both single-rate, which is developed in \cite{na02}, and multi-rate schemes are implemented for the above system and the results are compared.
The parameters in the control law~\eqref{claww} are set as $c_1 =6$ and $c_2=1$.
Simulations are carried out for different values of sampling times.
Note that for the single-rate scheme, the ZOH and sampler sampling times are set equal to $T_s$.
On the other hand, for the multi-rate scheme, these values are non-equal and denoted by $T$ and $T_s$, respectively.
Following the arguments in~\cite[Example III.1]{nl02}, in each set of simulations, an estimated region of initial conditions is obtained (\textit{i.e} $|x_0| \leq R_T$), for which the state $x$ stays bounded.
This region is denoted as region of boundedness (ROB).
The results are given in Table \ref{tablee}.
As it is clear, generally, for the same value of sampling time, the multi-rate systems guarantee a lager ROB, compared to the single-rate system.
In particular, by lowering the ZOH sampling time in multi-rate scheme, one can guarantee a larger ROB for the closed-loop system.
Indeed, raising the sampling period from $0.3$ to $0.34$ would render the single-rate case unstable, while this margin for both the multi-rate cases is $0.38$.
\begin{center}
\begin{table}[h]
	\centering
			\vspace{-0.3cm}
	\caption{\small Radius of ROB (\textit{i.e.} $R_T$) for different values of sampler ($T_s$) and ZOH ($T$) sampling times.}
	\resizebox{9cm}{!}{
		\begin{tabular}{lccccc}
			\hline
			& \multicolumn{5}{c}{$T_s$} \\ \cdashline{2-6} 
			&  0.1 & 0.2 & 0.3 & 0.34 & 0.38  \\ \cline{1-6} 
			SR Scheme & 24 & 11.5 & 7.1 & 0 & 0 \\    
			MR Scheme ($T=0.05$)  & 27.5 & 13.4 & 8.5 & 7   & 0   \\
			MR Scheme ($T=0.01$)  & 30.2 & 14.7 & 9.4 & 8.2 & 0   \\
			\hline \hline
			\end{tabular}}
		\vspace{-0.9cm}
			\label{tablee}
\end{table}
\end{center}
To compare the results of systems from a performance point of view, state trajectories are obtained for all three multi-rate and single-rate cases given in Table~\ref{tablee}, with $T_s=0.3 \ s$, and identical initial conditions $x_0 = [1.2,-5.9]^\top$. The state response is depicted in Figure~\ref{fig:multirate}, where the trajectories of the both multi-rate schemes look quite similar.
The single-rate scheme, however, deteriorates the state performance and leads to a larger overshoot and longer settling time. 
\begin{figure}[t]
	\begin{center}
			\includegraphics[width=14cm,height=9cm]{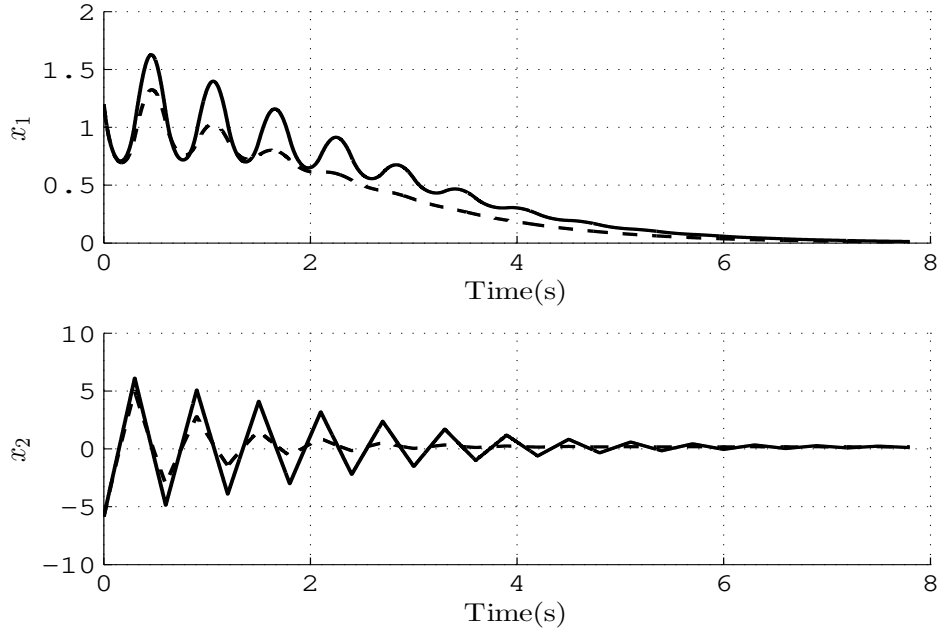}
			\caption{State response of the plant, controlled by two different multi-rate schemes (dashed) and single-rate scheme (solid).}
			\label{fig:multirate}
		\end{center}
		\vspace{-0.8cm}
\end{figure}
	
\section{Conclusions}\label{S:Conclusions}

In this paper, we investigated the properties of SP-iISS and SP-iIiSS for dual-rate nonlinear sampled-data systems.
In particular, we extended single-rate results in~\cite{na02} to the case of the low measurement rate constraint.
Towards this end, we compensated the lack of measurements using an estimator approximately reconstructing the current state.
Then we provided sufficient conditions under which a controller that semiglobally practically integral input-to-(integral-)state stabilizes the family of approximate discrete-time models of the single-rate nonlinear sampled-data system, also stabilizes the family of exact discrete-time models of the nonlinear sampled-data system in the same sense implemented in the dual-rate setting.
As a verification of our results, numerical simulations were given.

Our results can be further extended to accommodate output feedback control problem, transmission delays and message dropouts; e.g., see~\cite{Ustunturk2012,vnh12,pll08} for more details.


\appendices

\section{Technical Lemmas} \label{app:technical-lemmas}
\begin{lemma} \label{L:02}
Let any positive real numbers $(\Delta_x,\Delta_w,\varepsilon)$ be given. There exist $T^* > 0$ and $L > 0$ such that for any fixed $T \in (0,T^*)$ there exists $h^* \in (0,T]$ such that for all $h \in (0,h^*)$, $\abs{x(0)} \leq \Delta_x$ and $w \in \mathcal{L}_\infty (\Delta_w)$, the following holds
\begin{align}\nonumber
& \abs{x(k) - x_c (k)} \leq T \varepsilon + L \sum_{i=0}^{\ell-2} e^{L(i+1)T} \int_{(k-i - 1)T}^{(k-i)T}\abs{w(s)} \mathrm{d}s &
\end{align}
whenever $\max_{i \in \{ 0 , 1 , \dots , k \}} \abs{x(i)} \leq \Delta_x$ for some $k \in \Zp$.
\end{lemma}
\begin{proof}
The proof follows the same arguments as those in the proof of \cite[Claim 1]{lml08} and, hence, is omitted. 
\end{proof}
\begin{lemma} \label{L:03}
Let items~$\mathrm{\ref{item:iIiSS-app})}$ to~$\mathrm{\ref{item:iIiSS-lips})}$ of Theorem~\ref{thm:iIiSS} hold. Given any positive real numbers $(\Delta_x,\Delta_w,\delta)$, there exists $T^* > 0$ such that for each fixed $T \in (0,T^*)$ there exists some $h^* \in (0,T]$ such that
\begin{align}
& V_{T,h} (\mathcal{F}^e_{T,h} (x(k),w_T[k])) - V_{T,h}(x(k)) \leq T\! \left[ - \alpha(|x(k)|) \!+\! \frac{1}{T} \int_{kT}^{(k + 1)T} \!\!\!\!\!\!\!\!\!\!\! \hat\gamma(\abs{w(s)}) \mathrm{d}s \!+\! \delta \right]  & \label{eq:e23}
\end{align}
for all $h \in (0,h^*)$, all $\abs{x(k)} \leq \Delta_x$ and all $w \in \mathcal{L}_\infty (\Delta_w)$.
\end{lemma}
\begin{proof}
Let positive real numbers $(\Delta_x,\Delta_w,\delta)$ be given.
Denote $\Delta_1$ $:=$ $\max \{ \Delta_x$, $\underline\alpha^{-1} (\overline\alpha (\Delta_{x}) + 1) + 1 \}$, $\Delta_2 := \Delta_w$, $\Delta_3 := \tilde\gamma(\Delta_w)$ and $\delta_1 := \frac{\delta}{2}$.
Let $\Delta_1$ generate $T^*_1$, $h^*_1$ and $\tilde L$ from item~\ref{item:iIiSS-lips}).
Let the quadruple $(\Delta_1,\Delta_2,\Delta_3,\delta_1)$ generate $T^*_2$, $h^*_2$, $M$ and a SP-iIiSS Lyapunov function from item~\ref{item:iIiSS-app}) of Theorem~\ref{thm:iIiSS}.
Take any $\varepsilon_1 > 0$. Let the quadruple $(\Delta_1,\Delta_2,\varepsilon_1)$ generate $T^*_3$, $h^*_3$ and $L$ from Lemma~\ref{L:02}.
By item~\ref{item:iIiSS-lips}) of Theorem~\ref{thm:iIiSS}, there exist $T^*_4 > 0$, $h^*_4 > 0$ and $\Delta_u > 0$ such that $\abs{u_{T,h} (x)} \leq \Delta_u$ for all $x \in \Rn$ with $\abs{x} \leq \Delta_1$, all $T \in (0,T^*_4)$ and all $h \in (0,h^*_4)$.
Also, let $(\Delta_1,\Delta_u,\Delta_2)$ generate $\rho$, $T^*_5$ and $h^*_5$ from item~\ref{item:iIiSS-one-step}) of Theorem~\ref{thm:iIiSS}.
Take $T^*_6$ and $h^*_6$ such that
$$
M \left[ \rho(h^*_6) + \tilde{L} (e^{LT^*_6} - 1) \left( \varepsilon_1 + L e^{L(\ell-1)T_6^*} (\ell-1) \Delta_3 \right) \right] \leq \delta_1 .
$$
Let $T^*_7$ and $h_7^*$ be such that
$$
T^*_7 \Big[ \rho(h^*_7) + \tilde{L} (e^{LT^*_7} - 1) \left( \varepsilon_1 +  L e^{L(\ell-1)T^*_7} \Delta_3 \right) \Big] \leq 1 .
$$
Let $T^*_8$ be such that $T^*_8 (\hat\gamma (\Delta_2) + \delta_1) < 1$. Denote $T^* := \min \{ \min_{i\in\{1,\dots,8\}} T^*_i ,1 \}$ and $h^* := \min \{ \min_{i\in\{1,\dots,7\}} h_i^*, T^* \}$.

Pick an arbitrary $x(k)$ such that $\abs{x(k)} \leq \Delta_x$.
From the fact that $\mathcal{L}_\infty (\Delta_w) = \mathcal{L}_\infty (\Delta_2) \cap \mathcal{L}_{\hat\gamma} (\Delta_3)$ for all signals defined over a time interval of less than one unit of time, we have that for all $h \in (0,h^*_2)$ associated with each fixed $T \in (0,T^*_2)$, all $\abs{x(k)} \leq \Delta_x$ and all $w \in \mathcal{L}_\infty (\Delta_w)$
\begin{align*}
& V^a_{k+1} - V_k \leq T \left[ - \alpha(\abs{x(k)}) + \frac{1}{T} \int_{kT}^{(k+1)T} \hat\gamma(\abs{w(s)}) \mathrm{d}s + \delta_1 \right] . &
\end{align*}
So we get
\begin{align*}
V^e_{k+1} - V_k = & V^a_{k+1} - V_k + V^e_{k+1} - V^a_{k+1} & \\
\leq & T \left[ - \alpha(\abs{x(k)}) + \frac{1}{T} \int_{kT}^{(k+1)T} \hat\gamma(\abs{w(s)}) \mathrm{d}s + \delta_1 \right] + \abs{V^e_{k+1} - V^a_{k+1}} . &
\end{align*}
From item~\ref{item:iIiSS-app}) of Theorem~\ref{thm:iIiSS} and our choice of $T^*_8$, we have
\begin{align} \label{eq:e25}
\underline\alpha (\abs{\mathcal{F}^a_{T,h}(x,w_T)}) & \leq V^a_{k+1} \leq V_k + T (\hat\gamma (\Delta_2) + \delta_1) \leq \overline\alpha (\Delta_x) +1 . &
\end{align}
It follows from the definition of $\Delta_1$ that
\begin{align} \label{eq:fa-iiiss}
& \abs{\mathcal{F}^a_{T,h}(x,w_T)}< \Delta_1 . &
\end{align}
With the same arguments as those in proof of Theorem~\ref{thm:iISS} (cf.~\eqref{eq:fe-fes} and~\eqref{eq:fes-fas}), we have
\begin{align} \label{eq:e24}
& \abs{\mathcal{F}^e_{T,h}(x(k),w_T[k]) \! - \! \mathcal{F}^a_{T,h}(x(k),w_T[k])} \! \leq \! T (\rho(h) + \tilde{L} (e^{LT} - 1) \abs{x(k)-x_c(k)} ) . &
\end{align}
Applying Lemma~\ref{L:02} to the right-hand side of (\ref{eq:e24}) yields
\begin{align*}
& \abs{\mathcal{F}^e_{T,h}(x,w_T) - \mathcal{F}^a_{T,h}(x,w_T)} & \\
&\leq T \left( \rho(h) + \tilde{L} (e^{LT} - 1) \left( T \varepsilon_1 + L \sum_{i=0}^{\ell-2} e^{L(i+1)T} \int_{(k-i - 1)T}^{(k-i)T}\abs{w(s)} \mathrm{d}s \right) \right)  & \\
& \leq  T \left(\rho(h) + \tilde{L} (e^{LT} - 1) \left( T \varepsilon_1 + L e^{L(\ell-1)T} \sum_{i=0}^{\ell-2} \int_{(k-i - 1)T}^{(k-i)T}\abs{w(s)} \mathrm{d}s \right) \right) . &
\end{align*}
It follows from the fact that $w \in \mathcal{L}_\infty (\Delta_2)$ that
\begin{align*}
& \abs{\mathcal{F}^e_{T,h}(x,w_T) \! - \! \mathcal{F}^a_{T,h}(x,w_T)} \!\!\leq \! T \!\left(\!\rho(h) \! + \! \tilde{L} (e^{LT} \! - \! 1) \!\left( T \varepsilon_1 \!+\! L (\ell - 1) e^{L(\ell-1)T} \! \Delta_2 \right)\! \right) \! . &
\end{align*}
From the choice of $T^*_7$ and $h_7^*$, we obtain
\begin{align} \label{eq:e26}
& \abs{\mathcal{F}^e_{T,h}(x,w_T) - \mathcal{F}^a_{T,h}(x,w_T)} \leq 1 . &
\end{align}
Exploiting the last inequality of~\eqref{eq:e25} and~\eqref{eq:e26} gives
\begin{align}
\abs{\mathcal{F}^e_{T,h}(x,w_T)} & \leq \abs{\mathcal{F}^a_{T,h}(x,w_T)} + \abs{\mathcal{F}^e_{T,h}(x,w_T) - \mathcal{F}^a_{T,h}(x,w_T)}  \nonumber \\
& \leq \ul\alpha^{-1} (\ol\alpha (\Delta_x) + 1) + 1 \leq \Delta_1 . \label{eq:fe-iiiss}
\end{align}
By \eqref{eq:0909},~\eqref{eq:fa-iiiss} and the last inequality of~\eqref{eq:fe-iiiss}, we get
\begin{align*}
V^e_{k+1} - V_k \leq & T \left[ - \alpha(\abs{x(k)}) + \frac{1}{T} \int_{kT}^{(k+1)T} \hat\gamma(\abs{w(s)}) \mathrm{d}s + \delta_1 \right] & \nonumber \\
& + M \abs{\mathcal{F}^e_{T,h}(x(k),w_T[k]) - \mathcal{F}^a_{T,h}(x(k),w_T[k])} . &
\end{align*}
It follows from (\ref{eq:e24}) that
\begin{align*}
V^e_{k+1} - V_k \leq & T \left[ - \alpha(\abs{x(k)}) + \frac{1}{T} \int_{kT}^{(k+1)T} \hat\gamma(\abs{w(s)}) \mathrm{d}s + \delta_1 \right] & \nonumber\\
& + M \Big[ T\rho(h) +  \tilde{L} (e^{LT} - 1) \abs{x(k)-x_c(k)} \Big] .
\end{align*}
By Lemma~\ref{L:02}, we have
\begin{align*}
& V^e_{k+1} - V_k \leq T \left[ - \alpha(\abs{x(k)}) + \frac{1}{T} \int_{kT}^{(k+1)T} \hat\gamma(\abs{w(s)}) \mathrm{d}s + \delta_1 \right] & \\
& \;\; + M \Bigg[ T \rho(h) + \tilde{L} (e^{LT} - 1) \left( T \varepsilon_1 + L \sum_{i=0}^{\ell-2} e^{L(i+1)T} \int_{(k-i - 1)T}^{(k-i)T}\abs{w(s)} \mathrm{d}s \right) \Bigg] & \\
&\leq T \left[ - \alpha(\abs{x(k)}) + \frac{1}{T} \int_{kT}^{(k+1)T} \hat\gamma(\abs{w(s)}) \mathrm{d}s + \delta_1 \right] & \\
& \;\; + M \Bigg[ T \rho(h) + \tilde{L} (e^{LT} - 1) \left( T \varepsilon_1 + L e^{L(\ell-1)T} \sum_{i=0}^{\ell-2} \int_{(k-i - 1)T}^{(k-i)T}\abs{w(s)} \mathrm{d}s \right) \Bigg]. &
\end{align*}
It follows from the fact that $w \in \mathcal{L}_\infty (\Delta_2)$ that
\begin{align*}
V^e_{k+1} - V_k \leq & T \left[ - \alpha(\abs{x(k)}) + \frac{1}{T} \int_{kT}^{(k+1)T} \hat\gamma(\abs{w(s)}) \mathrm{d}s + \delta_1 \right] & \nonumber\\
& + T M \Bigg[ \rho(h) + \tilde{L} (e^{LT} - 1) \left( \varepsilon_1 + L e^{L(\ell-1)T} (\ell-1) \Delta_2 \right) \Bigg].
\end{align*}
By the choice of $T^*_6$ and $h^*_6$, we obtain
\begin{align*}
V^e_{k+1} - V_k \leq & T \left[ - \alpha(\abs{x(k)}) + \frac{1}{T} \int_{kT}^{(k+1)T} \hat\gamma(\abs{w(s)}) \mathrm{d}s + 2 \delta_1 \right] .
\end{align*}
This completes the proof.
\end{proof}

\end{document}